\documentclass{amsproc}
\usepackage{graphicx}
\usepackage{amssymb}
\usepackage{epstopdf}
\usepackage{amsmath,amsthm,amscd,amssymb}
\usepackage{latexsym}
\usepackage[colorlinks,citecolor=red,pagebackref,hypertexnames=false]{hyperref}
\usepackage{geometry}                
\numberwithin{equation}{section}

\theoremstyle{plain}
\newtheorem{theorem}{Theorem}[section]
\newtheorem{lemma}[theorem]{Lemma}

\theoremstyle{definition}
\newtheorem{definition}[theorem]{Definition}

\newtheorem{case[theorem]}{Case}

\def\dH{\dim_{{\mathcal H}}}
\def\R{\Bbb R}

\theoremstyle{remark}
\newtheorem{remark}[theorem]{Remark}

\numberwithin{equation}{section}

\newcommand{\abs}[1]{\lvert#1\rvert}

\begin{document}

\title{\parbox{14cm}{\centering{On radii of spheres determined by
      subsets of Euclidean space}}}


\author{Bochen Liu}
\address{Department of Mathematics \\ University of Rochester\\ Rochester, NY 14627}
\email{bliu19@z.rochester.edu}

\begin{abstract}  
In this paper we consider the problem of how large
the Hausdorff dimension of $E\subset\R^d$ needs to be in order to
ensure that the radii
set of $(d-1)$-dimensional spheres determined by $E$ has positive Lebesgue
measure. We also study the question of how often can a neighborhood of
a given radius repeat. We obtain two results. First, by applying a
general mechanism developed in \cite{mul} for studying
Falconer-type problems, we prove that a neighborhood of
a given radius cannot repeat more often than the statistical bound if
$\dH(E)>d-1+\frac{1}{d}$; In $\R^2$, the dimensional threshold is sharp. Second, by proving an
intersection theorem, we prove for a.e $a\in\R^d$, the radii set of
$(d-1)$-spheres with center $a$ determined by $E$ must have positive
Lebesgue measure if $\dH(E)>d-1$, which is a sharp bound for
this problem.
 
\end{abstract}
\maketitle

\section{Introduction}

\vskip.125in

The classical Falconer distance conjecture states that if 
a set $E \subset {\Bbb R}^d$, $d \ge 2$,
has Hausdorff dimension greater than $\frac{d}{2}$,
then the one-dimensional Lebesgue measure ${\mathcal L}^1(\Delta(E))$
of its distance set,
$$ \Delta(E):=\{|x-y|\in\Bbb R: x,y \in E\},$$ 
is positive, where $|\cdot|$ denotes the Euclidean distance. Falconer
gave an example based on the integer lattice showing that the exponent $\frac{d}{2}$ is best possible. The best results currently known, culminating almost three decades of efforts by Falconer \cite{Fal86}, Mattila \cite{Mat87}, Bourgain \cite{B94}, and others, are due to Wolff \cite{W99} for $d=2$  and  Erdo\v{g}an \cite{Erd05} for $d\ge 3$. They prove that 
${\mathcal L}^1(\Delta(E))>0$ if 
$$dim_{{\mathcal H}}(E)>\frac{d}{2}+\frac{1}{3}.$$

It is natural to consider related problems, where distances
determined by $E$ may be replaced by other geometric objects,
such as triangles \cite{tri}, volumes \cite{vol}, angles
\cite{ang1}, \cite{ang2}, and others. Generally, if $$\Phi:(\R^d)^{k+1}\rightarrow
\R^m$$ for some $1\leq m \leq \binom{d+1}{2}$, we
define a configuration set $$\Delta_\Phi(E)=\{\Phi(x^1,\dots,
x^{k+1}):x^j\in E\}$$ and ask how large $\dH(E)$ needs to be
to ensure $\mathcal{L}^m(\Delta_\Phi(E))>0$. For example, in the distance
problem, $k=1, \Phi(x^1,x^2)=|x^1-x^2|$; in the volume problem,
$k=d, \Phi(x^1,\dots,x^{d+1})=|\det(x^{d+1}-x^1,\dots,x^{d+1}-x^d)|$; in
the angle problem,
$k=1, \Phi(x^1,x^2)=\frac{x^1\cdot x^2}{|x^1||x^2|}$.
\vskip.125in
In this paper, we consider the problem of the distribution of radii of
spheres 
determined by $(d+1)$-tuples of points from $E$. 

\begin{definition}
We say a sphere is determined by a
$(d+1)$-tuple $(x^1,\dots,x^{d+1})$ if it is the unique sphere passing
through all the points $x^1,\dots,x^{d+1}$. We say a shpere is determined by a
set $E$ if it is determined by a
$(d+1)$-tuple $(x^1,\dots,x^{d+1})\in E\times\cdots\times E$.
\end{definition}

 Let
$R(x^1,\dots, x^{d+1})$ be the
radius of the unique 
$(d-1)$-dimensinal sphere determined by  $(x^1,\dots,x^{d+1})$
and it equals $0$ if such a sphere does not exist or it is not unique.  We obtain two types of results. First, we estimate how often a neighborhood of a given radius occurs, in the sense defined below. Second, we find the optimal dimension such that the Lebesgue measure of the set of radii is positive. Our main geometric results are the following.
\begin{theorem}\label{main}
Let $E\subset \R^d$ be compact and $\nu$ be a Frostman
measure on $E$. Then $\dH(E)>d-1+\frac{1}{d}$ implies
\begin{equation}\label{incieq}
(\nu\times\dots\times\nu)\{(x^1,\dots,x^{d+1}):\abs{R(x^1,\dots,x^{d+1})-t}<\epsilon\}\lesssim\epsilon
\end{equation}
where the implicit constant is uniform in $t$ on any bounded set.
\vskip.125in
Moreover, $\dH(E)>d-1+\frac{1}{d}$
implies $$\mathcal{L}(R(E\times\dots\times E))>0.$$ When $d=k=2$, this result is sharp in the sense that
\eqref{incieq} does not generally hold when $\dH(E)<\frac{3}{2}$.
\end{theorem}

In contrast to the incidence result, using Mattila's classical
estimate on the dimension of intersections (see Theorem \ref{Mat84}), one can prove
that 
\begin{equation}
  \label{rad=r}
\dH(E)>d-1\Longrightarrow\mathcal{L}(R(E\times\dots\times
E))=\R^+.  
\end{equation}

It's the optimal result because there is no unique $(d-1)$-dimensional
sphere passing through $d+1$ points in a hyperplane. However, one can never get this bound by improving the incidence
result above due to its sharpness.

\vskip.125in
We also prove an intersection result where rotations and translations
are replaced by dilations and translations (see Theorem
\ref{interthm}). From this intersection theorem we deduce the following.
\begin{theorem}\label{intercor}
Given $E\subset\R^d$ with $\dH(E)>d-1$, then for a.e. $a\in\R^d$,
$$\mathcal{L}(\{r>0:E\text{ determines } S^{d-1}_{a,r}\})>0,$$
where $S^{d-1}_{a,r}$ is the $(d-1)$-dimensional sphere with center $a$ and radius $r$.
\end{theorem}
The bound $d-1$ is sharp because a hyperplane can never determine a
$(d-1)$-dimensional sphere of finite radius.
\vskip.125in
{\bf Notation.} Throughout the paper, 

$X \lesssim Y$ means that there exists $C>0$
such that $X \leq CY$.

$S^{d-1}=\{x\in\R^d:|x|=1\}.$

For $A\subset\R^d$, $A_{a,r}=\{rx+a:x\in A\}$.

For a measure $\mu$ on $\R^d$ and a function $q(x)$ on $\R^d$, $q_*\mu$
is the measure on $q(\R^d)$ induced by $q(x)$.

$\hat{\mu}(\xi)=\int e^{-2\pi i x\cdot
  \xi}\,dx$ is the Fourier transform of the measure $\mu$.

\vskip.125in
\subsection*{Acknowledgement}
The author wishes to thank Professor Alex Iosevich for suggesting this
problem and comments that helped improve the manuscript.

\vskip.125in
\section{The Incidence Result}
\vskip.125in
In \cite{mul}, Grafakos, Greenleaf, Iosevich, and Palsson develop the
following general mechanism to solve Falconer-type problems.
\vskip.125in
We say $\Phi$ is translation invariant if it can be written
as $$\Phi_0(x^{k+1}-x^1,\dots,x^{k+1}-x^k).$$ 

\begin{theorem}[Grafakos, Greenleaf, Iosevich, and Palsson,
  2012] \label{t1} Suppose $\Phi$ is translation invariant and for
  some $\gamma>0$,
\begin{equation}\label{hatmu}
\begin{aligned}
\abs{\widehat{\mu_t}(-\xi,\xi,0,\dots,0)}&\lesssim
(1+\abs{\xi})^{-\gamma},\\
\abs{\widehat{\mu_t}(0,\xi,0,\dots,0)}&\lesssim
(1+\abs{\xi})^{-\gamma},\\
\abs{\widehat{\mu_t}(\xi,0,0,\dots,0)}&\lesssim
(1+\abs{\xi})^{-\gamma},
\end{aligned}
\end{equation}
where $\mu_t$ is the natural measure on
$\{u:\Phi_0(u)=t\}$ and the implicit constant
is uniform in $t$ on any bounded set. Then $\dH(E)>d-\frac{\gamma}{k}$ implies 
\begin{equation} \label{mineq}
(\nu\times\dots\times\nu)\{(x^1,\dots,x^{k+1}):\abs{\Phi(x^1,\dots,x^{k+1})-t}<\epsilon\}\lesssim\epsilon^m
\end{equation}
where $\nu$ is a Frostman measure on $E$ and the implicit constant
is uniform in $t$ in any bounded set. It follows that the $m$-dimensional Lebesgue measure
$\mathcal{L}^m(\Delta_\Phi (E))>0$. 
\end{theorem}
\vskip.125in
\section{Proof of Theorem \ref{main}}
\vskip.125in
Since $R(x^1,\dots,x^{d+1})$ is translation invariant and it can be
written as $R_0(x^{d+1}-x^1,\dots,x^{d+1}-x^d)$,  by Theorem
\ref{t1} it suffices to show all the three inequalities in
\eqref{hatmu} hold for the natural measure $\mu_1$ on $\{(u_0,\dots,u_d)\in\R^d:
R_0(u^1,\dots,u^d)=1\}$. Observe that
\begin{displaymath}
\begin{aligned}
&\ \text{a } (d+1)\text{-tuple } (x^1,\dots,x^{d+1}) \text{ determines a
  sphere with radius }1\\
\Leftrightarrow &\ (0, x^{d+1}-x^1,\dots,x^{d+1}-x^d) \text{ determines a
sphere with radius } 1 \text{ passing through the origin}\\
\Leftrightarrow&\ x^{d+1}-x^i\in a+S^{d-1} \text{ for some } a\in S^{d-1}
\text{ and } 0\neq x^{d+1}-x^i\neq x^{d+1}-x^j, \forall i\neq j.\\
\end{aligned}
\end{displaymath}

Hence
\begin{equation}\label{changevar}
\{(u^1,\dots,u^d):R_0(u^1,\dots,u^d)=1\}=\{(\sigma_0+\sigma_1,\sigma_0+\sigma_2,\dots,\sigma_0+\sigma_d):\sigma_i\in
S^{d-1}\}-N,
\end{equation}
where $N$ is a set of measure $0$.

Let $\nu$ be the natural probability measure on $S^{d-1}$ and $\psi$ be some
smooth cut-off function which may vary from line to line. By changing
variables as \eqref{changevar},
\begin{equation}\label{muhat}
\begin{aligned}
\abs{\widehat{\mu_1}(\xi,-\xi,0,\dots,0)}&=\left|\int_{\{u:R_0(u)=1\}} e^{-2\pi i
  u\cdot(\xi,-\xi,0,\dots,0)}\,du\right|\\&=\left|\int\dots\int e^{-2\pi
i\xi\cdot(\sigma_1-\sigma_2)} \psi\ d\nu(\sigma_0)\dots
d\nu(\sigma_d)\right|\\&\lesssim\left|\iint e^{-2\pi
i\xi\cdot(\sigma_1-\sigma_2)}\psi\ d\nu(\sigma_1)\,d\nu(\sigma_2)\right|\\&=\left|\int_{S^{d-1}} e^{-2\pi\xi\cdot\sigma_1}\left(\int_{S^{d-1}}
e^{2\pi
i\xi\cdot\sigma_2}\psi\ d\nu(\sigma_2)\right)d\nu(\sigma_1)\right|\\&=\left|\int_{S^{d-1}}
e^{-2\pi\xi\cdot\sigma_1} a(\xi,\sigma_1)\,d\nu(\sigma_1)\right|.
\end{aligned}
\end{equation}

By stationary phase \cite{So93},
$|a(\xi,\sigma_1)|\lesssim(1+|\xi|)^{-\frac{d-1}{2}}$ and
\begin{equation}
\left|\int_{S^{d-1}} e^{-2\pi\xi\cdot\sigma_1}\left((1+|\xi|)^{\frac{d-1}{2}}a(\xi,\sigma_1)\right)d\nu(\sigma_1)\right|\lesssim(1+|\xi|)^{-\frac{d-1}{2}}.
\end{equation}

Hence
$$\abs{\widehat{\mu_1}(\xi,-\xi,0,\dots,0)}\lesssim(1+|\xi|)^{-(d-1)}.$$
\vskip.125in
For $\widehat{\mu_1}(\xi,0,\dots,0)$ and $\widehat{\mu_1}(0,\xi,0,\dots,0)$,
  after changing variables as above we have
\begin{displaymath}
\begin{aligned}
\widehat{\mu_1}(\xi,0,\dots,0)&=\int\dots\int e^{-2\pi
i\xi\cdot(\sigma_0+\sigma_1)} \psi\ d\nu(\sigma_0)\dots
d\nu(\sigma_d),\\
\widehat{\mu_1}(0,\xi,0,\dots,0)&=\int\dots\int e^{-2\pi
i\xi\cdot(\sigma_0+\sigma_2)} \psi\ d\nu(\sigma_0)\dots
d\nu(\sigma_d).
\end{aligned}
\end{displaymath}
By a similar argument,
$$\abs{\widehat{\mu_1}(\xi,0,\dots,0)},\ \abs{\widehat{\mu_1}(0,\xi,0,\dots,0)}\lesssim(1+|\xi|)^{-(d-1)},$$
which completes the proof of Theorem \ref{main}.
\vskip.25in
\section{Sharpness of Theorem \ref{main} in $\R^2$}
\vskip.125in
When $d=2$, Theorem \ref{main} says $\dH(E)>\frac{3}{2}$ implies
\eqref{mineq}, more precisely
\begin{equation}\label{rmineq}
(\nu\times\nu\times\nu)\{(x^1,x^2,x^3):\abs{R(x^1,x^2,x^3)-t}<\epsilon\}\lesssim\epsilon
\end{equation}
where $\nu$ is a Frostman measure on a compact set $E\subset
\R^2$ and the implicit constant is uniform in $t$ on any bounded set. Now we will show $\frac{3}{2}$
is sharp for \eqref{rmineq} with a counterexample motivated by Mattila
\cite{Mat87}.
\vskip.125in
Without loss of generality, fix $t=100$. Let $C_\alpha\subset[0,1]$
denote the Cantor set of dimension
$\alpha$ and $\nu$ denote the natural probability measure on $C_\alpha$. Let 
\begin{equation}
E=(\cup_{k=-200}^{200}(C_\alpha+k))\times[-200,200]
\end{equation}
and extend $\nu$ to $E$ in the natural way.

\vskip.25in
\begin{center}
\includegraphics[height=2.8in]{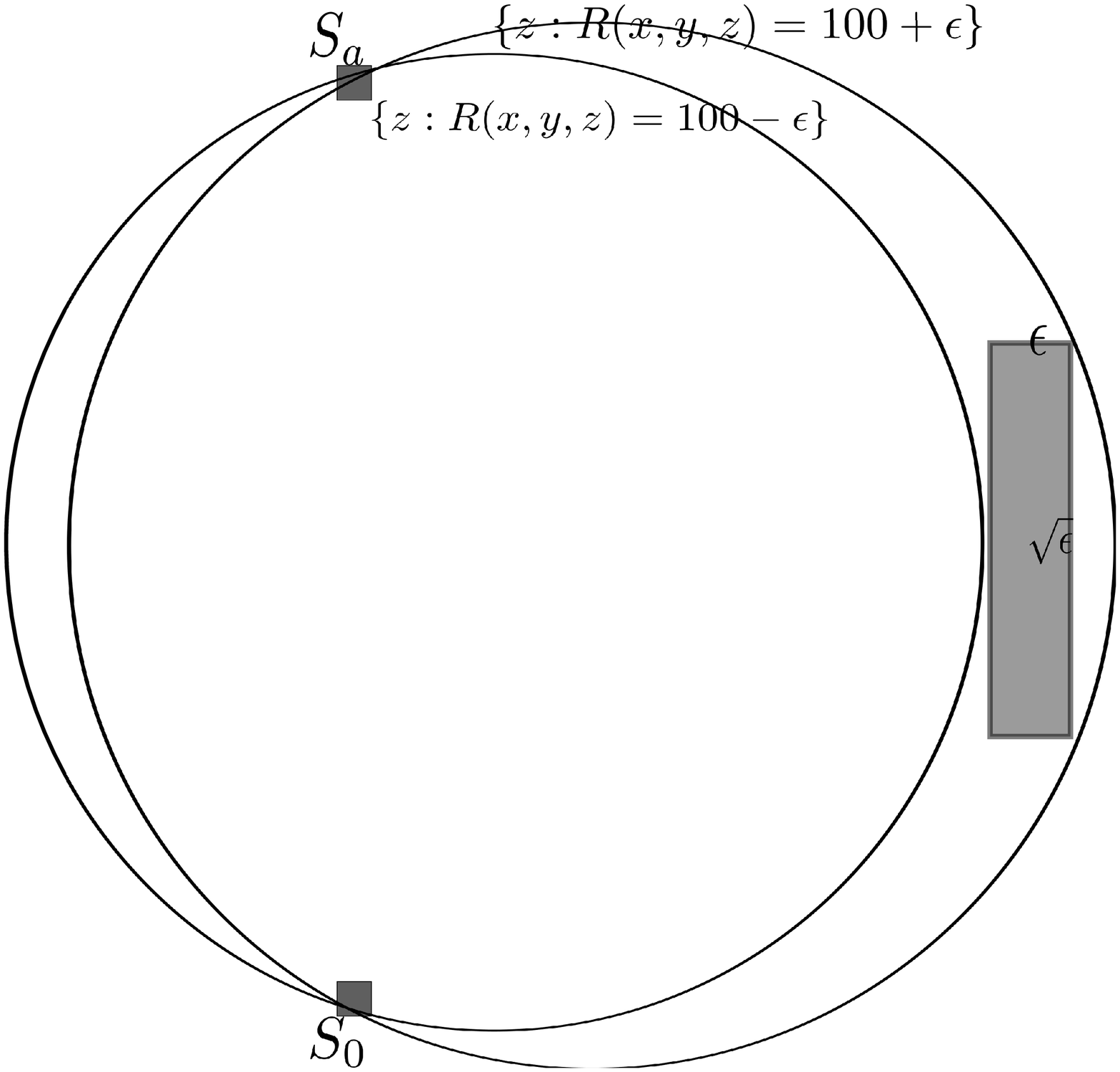}
\end{center}
\vskip.25in
Let $$S_a=\{(x,y)\in\R^2:0\leq x\leq 1, a\leq y\leq a+1\}.$$
\vskip.125in
When
$a\gg 1$, for all $x\in S_0,
y\in S_a$, the angle between the $y$-axis and the vector $y-x$ is very
small. Therefore  we can pick a $\sqrt{\epsilon}\times\epsilon$
rectangle in
$$\{z\in\R^2:\abs{R(x,y,z)-100}<\epsilon\}$$
such that its $\sqrt{\epsilon}$-edges are parallel
to the $y$-axis. (see the figure above). 
\vskip.125in
If $a$ changes continuously, the $\sqrt{\epsilon}\times\epsilon$
rectangle translates continuously without any rotation. Then there exsit an
$a_0\in(a-2,a)$ such that one of its $\sqrt{\epsilon}$-edges lies on a line $x=n$ for some $n\in\mathbb{Z}$. By the
construction of $E$, for all $x\in S_0, y\in S_{a_0}$,
\begin{equation}
\nu\{z\in\R^2:\abs{R(x,y,z)-100}<\epsilon\}\gtrsim\epsilon^{\frac{1}{2}}\cdot\epsilon^\alpha=\epsilon^{\frac{1}{2}+\alpha}.
\end{equation}

Hence
\begin{equation}\label{shineq}
\begin{aligned}
&\nu\times\nu\times\nu\{(x,y,z):\abs{R(x,y,z)-100}<\epsilon\}\\\gtrsim
&\nu\times\nu\times\nu\{(x,y,z):x\in S_0, y\in S_a,
\abs{R(x,y,z)-100}<\epsilon\}\\\gtrsim& \epsilon^{\frac{1}{2}+\alpha},
\end{aligned}
\end{equation}
which implies that when $\alpha<\frac{1}{2}$, \eqref{rmineq}
fails. 
\vskip.125in
It follows that for $\dH(E)<\frac{3}{2}$, \eqref{rmineq} is not generally
true, which proves the sharpness arguement of Theorem \ref{main}.

\begin{remark}
The sharpness of \eqref{rmineq} doesn't mean $\frac{3}{2}$ is the best
possible for radii problem, but for the method we use, it cannot be,
generally, improved.
\end{remark}

\begin{remark}
The example above cannot show the sharpness in higher
dimensions. For example, in $\R^3$, the rectangle we pick will
become a $\sqrt{\epsilon}\times\sqrt{\epsilon}\times\epsilon$ rectangle. If
we still construct a set like $C_\alpha\times C_\beta\times C_\gamma$, we
will have a lower bound
$\epsilon^{\frac{\alpha}{2}+\frac{\beta}{2}+\gamma}$ for
\eqref{mineq}. In nontrivial cases (i.e. $\dH(E)>2$) it is always true that
$\frac{\alpha}{2}+\frac{\beta}{2}+\gamma>\frac{1}{2}(\alpha+\beta+\gamma)>1$,
which does not contradict the upper bound. In fact, in many relavant problems, a
similar example can show the sharpness in $d=2$ but fail in $d\geq3$
(see, e.g. \cite{tri}, \cite{Mat87}).
\end{remark}
\vskip.25in
\section{An Intersection Theorem}
\vskip.125in
Given $A,B\subset \R^d$, we consider the behavior of the
intersection of $A\cap T_\alpha(B)$, where $T_\alpha$ is some family
of transformations. Mattila (\cite{Mat87}, \cite{Mat84}) proves a general
intersection theorem for orthogonal transformations. 
\begin{theorem}[Mattila, 1984]\label{Mat84}
In $\R^d$, let $s,t>0, s+t>d$ and $t>\frac{d+1}{2}$. If
$A,B\subset\R^d$ are Borel sets with $\mathcal{H}^s(A)>0,\mathcal{H}^t(B)>0$, then
for $\theta_d$ almost all $g\in O(d)$, 
$$\mathcal{L}(\{a\in\R^d:\dH(A\cap(gB+a)\geq s+t-d\})>0,$$
where $\theta_d$ is the Haar measure on $O(d)$.
\end{theorem}
There are also intersection theorems on larger transformation
groups, e.g. similarities (\cite{Kah86}, \cite{Mat84}).
\begin{theorem}[Kahane,1986]
Let $G$ be a closed subgroup of $GL(n,\R)$ and let $\tau$ be a Haar
measure. Let $E$ and $F$ be two $\sigma$-compact subsets of
$\R^d$. Then for $\tau$-almost all $g\in G$ and $\epsilon>0$,
$$\mathcal{L}(\{a\in\R^d:\dH(E\cap(gF+a))\geq\dH(E)+\dH(F)-d-\epsilon\})>0.$$
\end{theorem}
From
Theorem \ref{Mat84}, we know when $d=2$, for every $r>0$, there exists some $a\in\R^2$
such that $\dH(E\cap S_{a,r}^{d-1})\geq \dH(E)-1>0$. Hence
\eqref{rad=r} holds for $d=2$. In higher dimensions it is followed by the
following lemma.

\begin{lemma}\label{indu}
Suppose $F\subset S^{d-1}\subset\R^d$ and
$\dH(F)>d-2$, then $F$ uniquely determines $S^{d-1}$, i.e. $S^{d-1}$ is the only sphere
which contains all the points of $F$.
\end{lemma}
\begin{proof}
It is followed by induction. For $d=2$, it's trivial because $F$ has
at least three points.

 Suppose it holds in $\R^{k-1}$. For $F\subset S^{k-1}\subset\R^k$ with $\dH(F)>k-2$, there
exists a hyperplane $P$ such that $\dH(P\cap F)>k-3$ (see, e.g. \cite{M95}. Then $P\cap F$ determines the lower dimensional sphere $P\cap
S^{k-1}$. Hence $F$ determines $S^{k-1}$ because there is at least one
point in $F\backslash (P\cap F)$.
\end{proof}
Note all the intersection results above try to determine when the translation
set has positive Lebesgue measure. Now let us consider when the
dilation set has positive Lebesgue measure.
\vskip.125in
\begin{theorem}\label{interthm}
Suppose $E\subset\R^d$ is compact with $\dH(E)=s>1$ and
$\Gamma\subset\R^d$ is a smooth
hypersurface with nonzero Gaussian curvature.
Then
\begin{equation}
  \label{eq:inter}
  \mathcal{L}(\{r\in\R:\dH(E\cap \Gamma_{a,r})\})\geq
s-1\})>0
\end{equation}
holds for a.e. $a\in\{z\in\R^d:E\subset\bigcup_{r\in\R}\Gamma_{z,r}\}.$
\end{theorem}

In particularly, letting $\Gamma=S^{d-1}$, Theorem \ref{intercor} follows from Theorem \ref{interthm} and Lemma \ref{indu}.
\vskip.25in
\section{Proof of Theorem \ref{interthm}}
\vskip.125in
If $E\subset\bigcup_{r\in\R}\Gamma_{a,r}$ for some $a\in\R^d$, by the
compactness of $E$, it suffices to consider the case that $\Gamma$ is bounded and on any line passing through $a$ there is at
most one point of $\Gamma$.
\vskip.125in
Let $\sigma$ denote the
surface measure on $\Gamma$. Let $\{p_i\}$ be a partition of unity on
$\Gamma$ such that in the
support of each $p_i$, $\Gamma$ has a local coordinate system
$u^i=(u_1^i,\dots,u_{d-1}^i)$. Thus we have a well-defined coordinate
system for the
cone $C_i=\{rx: x\in \text{supp}\, p_i, r\in\R-\{0\}\}$. We can also extend
$p_i$ to $C_i$ by setting $p_i(rx)=p(x), x\in\Gamma$. By changing
variables $x=r\,x(u^i)$ in each $C_i$, we have
\begin{equation}\label{partition}
  \begin{aligned}
\int_{\bigcup_r
  \Gamma_{a,r}}f(x)\,dx&=\sum_i\int_{C_i}p_i(x-a)f(x)\,dx\\&=\sum_i\iint
p_i(r\,x(u^i))f(r\,x(u^i)+a)|r|^{d-1}\phi_i(u^i)\,du^i\,dr\\&=\sum_i\int_\R\int_{\text{supp}\,p_i}
p_i(x)f(rx+a)|r|^{d-1}\tilde{\phi}_i(x)\,d\sigma(x)\,dr\\&=\int_\R\int_\Gamma f(rx+a)|r|^{d-1}\psi(x)\,d\sigma(x)\,dr,
\end{aligned}
\end{equation}
where $\phi_i,\tilde{\phi}_i,\psi$ are smooth cut-off functions. 
\vskip.125in
Since $\dH(E)=s>1$, for every $\epsilon>0$, there exists a measure
$\mu$ on $E$ such that the $(s-\epsilon)$-energy $I_{s-\epsilon}(\mu)<\infty$. Let $q(x)=rx+a$. Define measures
$\sigma_{a,r}$ on $\Gamma_{a,r}$ and $\mu_{a,r}$ on $E\cap \Gamma_{a,r}$
by
\begin{equation}\label{measure}
\begin{aligned}
\sigma_{a,r}&=q_*(\psi\sigma),\\
\mu_{a,r}&=\lim_{\delta\rightarrow 0}\mu*\rho_\delta \,d\sigma_{a,r},
\end{aligned}
\end{equation}
where $\rho_\delta(x)=\delta^{-d}\rho(\frac{x}{\delta})$, $\rho\in
C_0^\infty$ and $\int\rho=1$.
\vskip.125in
Let
\begin{displaymath}
g(r)=\begin{cases}
|r|^{d-1}\Phi(r) &\text{if}\ -1<r<1\\
\Phi(r) &\text{otherwise}
\end{cases}
\end{displaymath}
where $\Phi\in L^1$ and $\Phi>0$ everywhere.

\begin{lemma}\label{lem2}
Under notations above,
\begin{equation}
\iint g(r)I_{s-1-\epsilon}(\mu_{a,r})\,dr\,da<\infty.
\end{equation}
\end{lemma}
 Denote
$\nu_{a,r}=\frac{g(r)^{\frac{1}{2}}\mu_{a,r}}{|r|^{d-1}\mu_{a,r}(\R^d)}$,
$G_a=\{r\in\R:\mu_{a,r}(\R^d)>0\}$ and
$H=\{z\in\R^d:E\subset\bigcup_{r\in\R}\Gamma_{z,r}\}$. From Lemma \ref{lem2}, $I_{s-1-\epsilon}(\mu_{a,r})<\infty$ for a.e. $(r,a)\in
\R\times\R^{d}$. Then by \eqref{partition}, \eqref{measure}, for $a\in H$,
\begin{equation}
\begin{aligned}
1&=(\mu(\R^d))^2\\&=\left(\lim_{\delta\rightarrow
  0}\int_{\bigcup_r\Gamma_{a,r}}\mu*\rho_\delta(x)\,dx\right)^2\\&=\left(\lim_{\delta\rightarrow
  0}\iint\mu*\rho_\delta(rx+a)\abs{r}^{d-1}\psi(x)\,d\sigma(x)\,dr\right)^2\\&=\left(\lim_{\delta\rightarrow
  0}\iint\mu*\rho_\delta(x)|r|^{d-1}\,d\sigma_{a,r}(x)\,dr\right)^2\\&=\left(\int |r|^{d-1}\mu_{a,r}(\R^d) \,dr\right)^2\\&=\left(\int_{G_a}|r|^{d-1}
\mu_{a,r}(\R^d)I_{s-1-\epsilon}^{-\frac{1}{2}}(\mu_{a,r})I_{s-1-\epsilon}^{\frac{1}{2}}(\mu_{a,r})\,dr\right)^2\\&=\left(\int_{G_a}
I_{s-1-\epsilon}^{-\frac{1}{2}}(\nu_{a,r})g(r)^{\frac{1}{2}}I_{s-1-\epsilon}^{-\frac{1}{2}}(\mu_{a,r})\,dr\right)^2\\&\leq\int_{G_a}
I_{s-1-\epsilon}^{-1}(\nu_{a,r})\,dr
\int_{G_a} g(r)I_{s-1-\epsilon}(\mu_{a,r})\,dr.
\end{aligned}
\end{equation}
Therefore,
\begin{equation}\label{lem1}
\begin{aligned}
\int_{H}\left(\int_{G_a} I_{s-1-\epsilon}^{-1}(\nu_{a,r})dr\right)^{-1}da&\leq
\int_H\int_{G_a} g(r)I_{s-1-\epsilon}(\mu_{a,r})\,dr\,da\\&\leq\iint
g(r)I_{s-1-\epsilon}(\mu_{a,r})\,dr\,da\\&<\infty.
\end{aligned}
\end{equation}

It follows that for a.e. $a\in\R^d$, $\int_{G_a}
I_{s-1-\epsilon}^{-1}(\nu_{a,r})dr>0$. Hence for a.e. $a\in\R^d$,
$\mathcal{L}(G_a)>0$ and $\dH(E\cap\Gamma_{a,r})\geq s-1-\epsilon$ for
a.e. $r\in G_a$. Since $G_a$ is independent of $\epsilon$, by choosing a sequence $\epsilon_j\rightarrow0$, we
complete the proof of Theorem \ref{interthm}. 
\vskip.25in
\section{Proof of Lemma \ref{lem2}}
\vskip.125in
 From the the well-known equality
$I_\alpha(\mu)=c_{\alpha,d}\int
\abs{\hat{\mu}(\xi)}\abs{\xi}^{-d+\alpha}\,d\xi$ (see,
e.g. \cite{M95}) and Plancherel, we have
\begin{equation}\label{mainineq1}
\begin{aligned}
&\iint g(r)I_{s-1-\epsilon}(\mu_{a,r})\,dr\,da\\=&c_{\alpha,d}\iiint
g(r)\left|\hat{\mu}*\widehat{\sigma_{a,r}}(\xi)\right|^2|\xi|^{-d-1+s-\epsilon}\,d\xi
\,dr\,da
\\=&c_{\alpha,d}\iiint
g(r)\left|\int\hat{\mu}(\eta)\widehat{\psi\sigma}(r(\xi-\eta)) e^{-2\pi i a\cdot(\xi-\eta)}\,d\eta\right|^2|\xi|^{-d-1+s-\epsilon}\,d\xi
\,dr\,da\\=&c_{\alpha,d}\iint\left(\int
\left|\int\hat{\mu}(\eta)\widehat{\psi\sigma}(r(\xi-\eta)) e^{2\pi i
    a\cdot\eta}\,d\eta\right|^2da\right) \ g(r)|\xi|^{-d-1+s-\epsilon}\,d\xi \,dr
\\=&c_{\alpha,d}\iint\left(\int
\left|\hat{\mu}(a)\widehat{\psi\sigma}(r(\xi-a))\right|^2da\right)g(r)|\xi|^{-d-1+s-\epsilon}\,d\xi \,dr
\end{aligned}
\end{equation}

Since $\Gamma$ is smooth with nonzero Gaussian curvature everywhere, by
stationary phase \cite{St93} and
the construction of $g$, 
$$\int|\widehat{\psi\sigma}(r(\xi-a))|^2g(r)\,dr\lesssim|\xi-a|^{-(d-1)}.$$
Thus, to prove Lemma \ref{lem2}, it suffices to show
\begin{equation}\label{mainineq2}
\begin{aligned}
\iint|\hat{\mu}(a)|^2|\xi-a|^{-(d-1)}|\xi|^{-d-1+s-\epsilon}\,da\,d\xi<\infty.
\end{aligned}
\end{equation}

\vskip.125in
For each $a\neq0$, let $\xi=|a|\zeta$, then
$$\int|\xi-a|^{-(d-1)}|\xi|^{-d-1+s-\epsilon}\,d\xi=|a|^{-d+s-\epsilon}\int|\zeta-\frac{a}{|a|}|^{-(d-1)}|\zeta|^{-d-1+s-\epsilon}\,d\zeta.$$

When $\epsilon$ is small, $1<s-\epsilon<d$, then $\int|\zeta-\frac{a}{|a|}|^{-(d-1)}|\zeta|^{-d-1+s-\epsilon}\,d\zeta<\infty$
uniformly because $\frac{a}{|a|}\in S^{d-1}$ which is compact. Hence
\begin{equation}
\begin{aligned}
\iint|\hat{\mu}(a)|^2|\xi-a|^{-(d-1)}|\xi|^{-d-1+s-\epsilon}\,da\,d\xi&\lesssim\int|\hat{\mu}(a)|^2|a|^{-d+s-\epsilon}\,da\\&<\infty,
\end{aligned}
\end{equation}
which proves \eqref{mainineq2} and completes the proof of Lemma \ref{lem2}.

\end{document}